\documentclass[11pt]{article}

\usepackage[utf8]{inputenc}
\usepackage[english]{babel}
\usepackage{amsthm}
\usepackage{amsmath}
\usepackage{amssymb}
\usepackage{enumerate}
\usepackage{cite}
\usepackage{graphicx}
\usepackage{tikz}

\theoremstyle{plain}
\newtheorem{theorem}{Theorem}

\newtheorem{corollary}[theorem]{Corollary}

\theoremstyle{definition}

\theoremstyle{remark}
\newtheorem{remark}[theorem]{Remark}

\newcommand{\F}{\mathbb{F}}
\newcommand{\fp}{\F_p}

\title{On the pseudorandomness of automatic sequences}

\author{L\'aszl\'o M\'erai and Arne Winterhof\\
  \\ \small
  Johann Radon Institute for Computational and Applied Mathematics\\ \small
  Austrian Academy of Sciences\\ \small
  Altenbergerstr.\ 69,
  4040 Linz, Austria\\ \small
  \texttt{\{laszlo.merai,arne.winterhof\}@oeaw.ac.at}
}

\date{}

\begin{document}

\maketitle

\let\thefootnote\relax\footnote{The authors are partially supported by the Austrian Science Fund FWF Project 5511-N26 which is part of the Special Research Program ''Quasi-Monte Carlo Methods: Theory and Applications''.\\
“This is a pre-print of an article published in Cryptography and Communications. The final authenticated version is available online at: https://doi.org/10.1007/s12095-017-0260-7”.
}

\begin{abstract}
We study the pseudorandomness of automatic sequences in terms of well-distribution and correlation measure of order 2. We detect non-random behavior which can be derived either from the functional equations satisfied by their generating functions or from their generating finite automatons, respectively.
\end{abstract}




\section{Introduction}

Let $k\geq 2$ be an integer. A $k$-automatic sequence $(s_n)$ over an alphabet $\Sigma$ is the output sequence of a finite automaton, where the input is the $k$-ary digital expansion of $n$.
Automatic sequences have gained much attention during the last decades. For monographs and surveys about automatic sequences we refer to \cite{al,alsh,dr,RecurenceSequences}.

For a prime $k=p$, $p$-automatic sequences $(s_n)$ over the finite field $\F_p$ of $p$ elements can be characterized by a result of Christol \cite{christol}, see also \cite{chka}:
Let 
\[
G(x)=\sum_{n=0}^{\infty}s_nx^n
\]
be the {\em generating function} of the sequence $(s_n)$ over $\F_p$. 
Then $(s_n)$ is {\em $p$-automatic over $\F_p$} if and only if $G(x)$ is algebraic over $\fp[x]$, that is, there is a characteristic polynomial 
$0\ne h(x,y)\in \F_p[x,y]$ such that $h(x,G(x))=0$.

For example, the {\em Thue-Morse sequence} over $\F_2$ is defined by
\[
t_n=\left\{ \begin{array}{cl} t_{n/2} & \mbox{if $n$ is even},\\ t_{(n-1)/2}+1 & \mbox{if $n$ is odd},
              \end{array}\right.\quad n=1,2,\ldots
\]
with initial value $t_0=0$. Taking
\[
h(x,y)=(x+1)^3 y^2 +(x+1)^2y+x,
\]
its generating function $G(x)$ satisfies $h(x,G(x))=0$.
 
Any $p$-automatic sequence over $\F_p$ which is not ultimately periodic, that is, its generating function $G(x)$ is not rational, passes some unpredictability tests. 
In particular, it has large linear complexity profile. This can be expressed in terms of the local degrees of $h(x,y)$. 

We recall that the \textit{$N$th linear complexity} $L(s_n, N)$ of a sequence $(s_n)$ over $\F_p$ is the length~$L$ of a shortest linear recurrence relation satisfied by the first $N$ elements of $(s_n)$:
\[
 s_{n+L}=c_{L-1}s_{n+L-1}+\dots +c_1s_{n+1}+c_0s_n, \quad 0\leq n\leq N-L-1,
\]
for some $c_0,\ldots,c_{L-1}\in \F_q$.  
We use the convention that $L(s_n,N)=0$ if the first $N$ elements of $(s_n)$ are all zero and $L(s_n,N)=N$ if $s_0=\dots=s_{N-2}=0\ne s_{N-1}$.

For a random sequence $(s_n)$ we have
\begin{equation}\label{eq:random}
 L(s_n,N)=\frac{N}{2}+O(\log N) \quad \text{for all } N\geq 2,
\end{equation}
see \cite{Niederreiter88}.

The authors proved in \cite{MeraiWinterhof2016+} that any $p$-automatic sequence over $\F_p$ 
which is not ultimately periodic has $N$th linear complexity of (best possible) order of magnitude $N$,  that is constant times $N$, where the implied constant depends on the degree of a characteristic polynomial $h(x,y)$ of $G(x)$.
 
Especially, the $N$th linear complexity of sequences with $h(x,G(x))=0$ and local degree 2 in $y$ of $h(x,y)$ satisfy \eqref{eq:random}. For example, \cite[Theorem 1]{MeraiWinterhof2016+} applied to the Thue-Morse sequence gives
$$\left\lceil\frac{N-1}{2}\right\rceil\le L(t_n,N)\le \left\lceil\frac{N-1}{2}\right\rceil+1.$$
(The exact value $L(t_n,N)=2\left\lfloor \frac{N+2}{4}\right\rfloor$ can be obtained using a different method, see also \cite{MeraiWinterhof2016+}.)

Although automatic sequences have large linear complexity profile, they are statistically distinguishable from random sequences if $N$ is sufficiently large and certain pseudorandom measures are of (worst possible) order of magnitude $N$.

For a given finite sequence $(s_n)$ over $\F_2$ write
\[
U\left(s_n,t,a,b\right)=\sum_{j=0}^{t-1} (-1)^{s_{a+jb}},
\]
and for $D=(d_1, \dots, d_k)$ with non-negative integers $0\leq d_1<\dots <d_k$ write
\[
V\left(s_n,M,D\right)=\sum_{n=0}^{M-1} (-1)^{s_{n+d_1}+s_{n+d_2}+\dots +s_{n+d_{k}}}.
\]
Then the  \textit{$N$th well-distribution measure} of $(s_n)$ is 
\[
W\left(s_n,N\right)=
\max _{a,b,t} \left| U\left(s_n,t,a,b\right)\right|=
\max _{a,b,t} \left|
\sum_{j=0}^{t-1} (-1)^{s_{a+jb}} \right|,
\]
where the maximum is taken over all $a,b,t \in \mathbb{N}$ such that
$0 \leq a \leq a+(t- 1)b< N$, and the  \textit{$N$th correlation measure
of order $k$} of $(s_n)$ is 
\[
C_{k}\left(s_n,N\right)=\max_{M,D}\left| V\left(s_n,M,D\right) \right|=
\max_{M,D}\left|
\sum_{n=0}^M (-1)^{s_{n+d_1}+s_{n+d_2}+\dots +s_{n+d_{k}}} \right|,
\]
where the maximum is taken over all $D$ and $M$
such that $ d_{k} +M < N$. For more background on pseudorandom measures see \cite{Gyarmati-survey,sa,towi}.

The sequence $(s_n)$ possesses good properties of pseudorandomness
if both these measures $W\left(s_n,N\right)$ and $C_{k}\left(s_n,N\right)$ (at least for small
$k$) are `small' in terms of $N$ (in particular, both are $o(N)$ as
$N \rightarrow \infty$).
This terminology is justified since for a truly random sequence
$(s_n)_{n=0}^{N-1}$ each of these measures is $ N^{1/2} (\log N)^{O(1)}$. (For a more
precise version of this result see \cite{AKMMR}.) The Legendre sequence is an example of such a pseudorandom sequence with both small well-distribution and correlation measures, see \cite{masa-legendre}.

%

In Section~\ref{well} we show that a certain family of $2$-automatic sequences classified by its functional equation for its generating function $h(x,G(x))=0$ suffers a large well-distribution measure.
This family includes the Baum-Sweet sequence and the characteristic sequence of the set of sums of three integer squares.
In Section~\ref{corr} we show that another family, again characterized by $h(x,G(x))=0$ for a certain class of polynomials $h(x,y)$, is of large correlation measure of order 2. 
This family includes pattern sequences such as the Thue-Morse sequence
as well as the Rudin-Shapiro sequence and the regular paperfolding sequence. 

Note that it is known that both the Thue-Morse sequence and the Rudin-Shapiro sequence have a large correlation measure of order $2$, see \cite{masa}. 

In Section~\ref{sec:general} we prove another bound on the correlation measure of order 2 of any 2-automatic sequence in terms of the number of states of its generating finite automaton. Roughly speaking, if the number of states of the finite automaton is small, the correlation measure of order 2 of the corresponding sequence is large. However, the results of Section~\ref{corr} are slightly stronger but apply only to some special automatic sequences.

On the other hand, our last result implies that for any automatic sequence
with small correlation measure of order $2$, its state complexity has to be
large. In particular, we apply this result to the Legendre sequence.

\section{Sequences with large well-distribution measure}
\label{well}

First we mention two sequences with large well-distribution measure, the Baum-Sweet sequence and the characteristic sequence of the set of sums of three squares. 
Then we show that these sequences belong to a larger family of sequences with a certain type of polynomials $h(x,y)$ with 
$h(x,G(x))=0$ which all have a large well-distribution measure. 

\subsubsection*{Baum-Sweet sequence}
The Baum-Sweet sequence $(b_n)$ is a $2$-automatic sequence defined by the rule $b_0=1$ and for $n\ge 1$
$$
 b_n=
 \left\{ 
 \begin{array}{cl}
  1& \text{if the binary representation of $n$ contains no block of} \\
   &  \text{consecutive $0$'s of odd length,}\\
  0& \text{otherwise.}
 \end{array}
 \right.
$$
Equivalently, we have for $n\ge 1$ of the form $n=4^km$ with $m$ not divisible by $4$
\begin{equation*}
b_n=\left\{\begin{array}{ll} 0 & \mbox{if $m$ is even},\\ b_{(m-1)/2} & \mbox{if $m$ is odd}.\end{array}\right.
\end{equation*}

The Baum-Sweet sequence satisfies
\begin{equation}\label{bs}
  W(b_n,N)\geq \left|\sum_{n=0}^{\lfloor (N-3)/4\rfloor} (-1)^{b_{4n+2}}\right|=\left\lfloor \frac{N+1}{4}\right\rfloor \quad \text{for } N\geq 1.
\end{equation}
The generating function $G(x)$ of $(b_n)$ satisfies 
$$h(x,G(x))=0 \quad\mbox{with} \quad h(x,y)=y^4+xy^2+y.$$ 

\subsubsection*{The characteristic sequence of the set of sums of three
squares}

Let $(u_n)$ be the characteristic sequence of non-negative integers that can be written as a sum of three squares
\[
 u_n=
 \left\{
 \begin{array}{cl}
  1 & \text{if $n=u^2+v^2+w^2$ for some integers $u,v,w$,}\\
  0 & \text{otherwise.}
 \end{array}
 \right.
\]
By the Three-Square Theorem this is equivalent to
\[
 u_n=
 \left\{
 \begin{array}{cl}
  0 & \text{if there exist non-negative integers $a,k$ with $n=4^a(8k+ 7)$,}\\
  1 & \text{otherwise.}
 \end{array}
 \right.
\]
We have
\begin{equation}\label{cs} 
W(u_n,N)\geq \left|\sum_{n=0}^{\lfloor N/8\rfloor -1} (-1)^{u_{8n+7}}\right|=\left\lfloor \frac{N}{8}\right\rfloor\quad \text{for } N\geq 1.
\end{equation}
The generating function $G(x)$ of $(u_n)$ satisfies $h(x,G(x))=0$ with $$h(x,y)=(x^8+1)y^4+(x^8+1)y+x^6+x^5+x^3+x^2+x,$$
see \cite{howi}.

We present some generalizations of $(\ref{bs})$ and $(\ref{cs})$. 
\begin{theorem}
Let $(s_n)$ be a sequence over $\F_2$ with generating function $G(X)$ satisfying $h(x,G(x))=0$ for some 
polynomial $h(x,y)$ over $\F_2$ of the form
\begin{equation}\label{form1} 
h(x,y)=f_2(y^{2^\ell})+xf_1(y^2)+y+f_0(x)
\end{equation}
with polynomials $f_0,f_1,f_2$ over $\F_2$, $\deg f_0\le 2^\ell-3$, and $\ell\ge 2$. 
Then we have 
$$W(s_n,N)\ge \left\lfloor \frac{N+1}{2^\ell}\right\rfloor.$$
If $h(x,y)$ is of the form
\begin{equation}\label{form2}
h(x,y)=f_1(x^2,y^2)+(x^{2^\ell}+1)y+f_0(x)
\end{equation}
with polynomials $f_1$ and $f_0$ over $\F_2$ with $\deg f_0\le 2^\ell-2$ and $\ell\ge 1$,
then we have
$$W(s_n,N)\ge \left\lfloor \frac{N}{2^\ell}\right\rfloor.$$
\end{theorem}

\begin{proof}
First let $h(x,y)$ be of the form (\ref{form1}).
Comparing the coefficients of $h(x,G(x))=0$ at $x^{2^\ell n+2^\ell-2}$ we see that $s_{2^\ell n+2^\ell-2}=0$. Hence, 
$$\sum_{n=0}^{M-1} (-1)^{s_{2^\ell n+2^\ell-2}}=M.$$
Taking 
$$M=\left\lfloor \frac{N+1}{2^\ell}\right\rfloor$$
gives the first result.

If $h(x,y)$ is of the form (\ref{form2}), we compare the coefficients at $x^{2^\ell n+2^\ell-1}$ and get
$s_{2^\ell-1}=0$ and $s_{2^\ell (n+1)+2^\ell-1}=s_{2^\ell n+2^\ell-1}$ for $n\ge 0$ and the result follows analogously.
\end{proof}

\section{Large correlation measure of order 2 obtained from a characteristic polynomial}
\label{corr}

Now we prove a lower bound on the correlation measure of order $2$ for a large class of automatic sequences.

\begin{theorem}\label{thm:2}
 Let $(s_n)$ be a sequence over $\F_2$ with generating function $G(x)$ satisfying 
 $$h(x,G(x))=0$$ for some polynomial $h(x,y)$ of the form
\[
h(x,y)=(x+1)^{2^\ell}((a_1x+a_0)y^2+ y)+ f(x) 
\]
with $\ell \ge 0$, $\deg f \le 2^\ell-1$, and $(a_1,a_0)\neq(0,0)$.  
Then we have 
$$C_2(s_n,N)>\frac{N}{2^\ell+2}-2\quad \mbox{for }N\ge 2^{\ell+1}+4.$$
\end{theorem}

\begin{proof}
For any $k\ge 0$ comparing coefficients in $(x+1)^{2^{k+\ell}-2^\ell}h(x,G(x))=0$ at $x^{2n+2^{k+\ell}}$ and $x^{2n+2^{k+\ell}+1}$ provides
\begin{equation}\label{even} s_{2n}+s_{2n+2^{k+\ell}}=a_0(s_n+s_{n+2^{k+\ell-1}}),\quad n\ge 0, 
\end{equation}
and
\begin{equation}\label{odd} s_{2n+1}+s_{2n+2^{k+\ell}+1}=a_1(s_n+s_{n+2^{k+\ell-1}}),\quad n\ge 0,
\end{equation}
respectively.
 
We define the integer $M$ by $2^M\le \frac{N}{2^\ell+1}<2^{M+1}$ and put
for $k=0,1,\ldots,M$
$$\gamma_k=\sum_{n=0}^{2^k-1} (-1)^{s_n+s_{n+2^{k+\ell}}}.$$
By $(\ref{even})$ and $(\ref{odd})$ we get for $1\le k\le M$
\begin{eqnarray*}\gamma_k&=&\sum_{n=0}^{2^{k-1}-1} (-1)^{s_{2n}+s_{2n+2^{k+\ell}}}+\sum_{n=0}^{2^{k-1}-1} (-1)^{s_{2n+1}+s_{2n+2^{k+\ell}+1}}\\
&=&\sum_{n=0}^{2^{k-1}-1} (-1)^{a_0(s_{n}+s_{n+2^{k+\ell-1}})}+\sum_{n=0}^{2^{k-1}-1} (-1)^{a_1(s_{2n+1}+s_{n+2^{k+\ell-1}})}\\
&=&(a_0+a_1)\gamma_{k-1}+(2-a_0-a_1)2^{k-1},
\end{eqnarray*}
$\gamma_0=\pm 1$,
and thus by induction
$$\gamma_k=(a_0+a_1)^k\gamma_0+(2-a_0-a_1)(2^k-1)\quad \text{for } k=1,\ldots,M.$$
Hence, $C_2(s_n,N)\ge |\gamma_M|=2^M>\frac{N}{2^\ell+2}$ if $a_0=a_1=1$ and
$C_2(s_n,N)\ge |\gamma_M|\ge 2^M-2>\frac{N}{2^\ell+2}-2$ otherwise.
\end{proof}

\begin{remark}
The method can be extended to larger families of polynomials $h(x,y)$. However, for the readability we chose a simple family which covers all of the following examples.
\end{remark}

\subsection*{Examples}
\subsubsection*{Pattern sequence}
For a pattern $P\neq(0,\dots,0)$ of length $\ell$ define the sequence $(r_n)$ by
\begin{equation*}
 r_n\equiv e_P(n) \mod 2, \quad r_n\in\F_2,\quad n=0,1,\dots
\end{equation*}
where $e_P(n)$ is the number of occurrences of $P$ in the binary expansion of $n$. 
The sequence $(r_n)$ over $\F_2$ satisfies the following recurrence relation
\begin{equation}\label{eq:recurence}
 r_n=
 \left\{
 \begin{array}{cl}
 r_{\lfloor n/2\rfloor}+1  & \text{if } n\equiv a \mod 2^\ell,\\
 r_{\lfloor n/2\rfloor}  & \text{otherwise,}
 \end{array}
 \right.
 n=1,2,\dots
\end{equation}
with initial value $r_0=0$, where $a$ is the integer $0< a <2^\ell$ such that its  binary expansion corresponds to the pattern $P$.

Classical examples for binary pattern sequences are the {\em Thue-Morse sequence} ($\ell=1$ and $P=1$ ($a=1$)) and the {\em Rudin-Shapiro sequence} ($\ell=2$ and $P=11$ ($a=3$)).

\begin{corollary}
Let $a,\ell$ be integers with $1\leq a<2^\ell$. If $(r_n)$ is the pattern sequence defined by \eqref{eq:recurence}, then
\[
 C_2(r_n,N)>\frac{N}{2^\ell+2}-2 \quad \text{for } N\geq 2^{\ell+1}+4.
\]
\end{corollary}
\begin{proof}
The result follows from Theorem~\ref{thm:2} with $h(x,y)=(x+1)^{2^\ell+1}y^2+(x+1)^{2^\ell}y+x^a$.
\end{proof}

\subsubsection*{Regular paperfolding sequence}
The value of any given term $v_n\in \F_2$ in the regular paperfolding sequence can be defined as follows. If $n = m\cdot2^k$ where $m$ is odd, then
\begin{equation*}
 v_n=
 \left\{ 
 \begin{array}{cl}
  1& \text{if } m\equiv 1 \mod 4,\\
  0& \text{if } m\equiv 3 \mod 4,
 \end{array}
 \right. n=1,2,\ldots
\end{equation*}
and any $v_0\in \F_2$.

\begin{corollary}
 Let $(v_n)$ be the regular paperfolding sequence. Then
 \[
  C_2(v_n,N)>\frac{N}{6}-2 \quad \text{for } N\geq 12.
\]
\end{corollary}
\begin{proof}
The result follows from Theorem~\ref{thm:2} with $h(x,y)=(x+1)^4(y^2+y)+1$.
\end{proof}

\subsubsection*{A sequence with perfect lattice profile and perfect linear complexity profile}
The generating function $G(x)$ of the sequence $(w_n)$ over $\F_2$ defined by 
\begin{equation}\label{def:w}
w_{2n}=1\quad \mbox{and} \quad w_{2n+1}=w_n+1,\quad n=0,1,\ldots
\end{equation}
satisfies the functional equation $h(x,G(x))=0$ with $h(x,y)=(x+1)(xy^2+y)+1\in\F_2[x,y]$.
This is the only sequence with both a perfect linear complexity profile and a perfect 'lattice profile', see \cite{domewi} for more details. 
Sequences with the first are characterized by $w_0=1$ and $w_{2n+2}=w_{2n+1}+w_n$ but the choice of $w_{2n+1}$ is free for $n\ge 1$, see \cite{Niederreiter88-b}. 
Sequences with the latter are characterized by $w_{2n+1}=w_n+1$ but the choice of any $w_{2n}$ is free, see \cite{domewi}.

\begin{corollary}
 Let $(w_n)$ be the sequence defined by \eqref{def:w}. Then
 \[
  C_2(w_n,N)>\frac{N}{3}-2 \quad \text{for } N\geq 6.
\]
\end{corollary}

\section{Large correlation measures of order 2 obtained from the generating automaton}\label{sec:general}

In this section we prove lower bounds on the correlation measure of order 2 of automatic sequences in terms of the number of states of the automaton which generates the sequence. 

We recall that a \emph{finite automaton} is defined to be a 6-tuple $M=(Q,\Sigma,\delta,\allowbreak q_0,      \allowbreak \Delta,\tau)$, where $Q$ is a finite set of states, $\Sigma$ is the finite input alphabet, $\delta:Q\times\Sigma\rightarrow Q$ is the transition function, $q_0\in Q$ is the initial state, $\Delta$ is the output alphabet and $\tau:Q \rightarrow \Delta$ is the output function. As usual, we define $\delta(q,xa)=\delta(\delta(q,x),a)$ for all $q\in Q$, $x\in\Sigma^*$ and $a\in\Sigma$. For $k\geq 2$ put $\Sigma_k=\{0,1,\dots,k-1\}$.  Then we say that the sequence $(s_n)$ over a finite alphabet $\Delta$ is \emph{$k$-automatic} if there exists an automaton $(Q,\Sigma_k,\delta,q_0,\Delta,\tau)$ such that  $s_n=\tau(\delta(q_0,(n)_k))$ for all $n\geq 0$ where $(n)_k\in\Sigma_k^*$ is the word consisting of the $k$-ary digits of $n$.

In the following theorem we prove that an automatic sequence which is generated by an automaton with only a few states cannot have good pseudorandomness properties in terms of the correlation measure of order 2.

\begin{theorem}\label{thm:general}
Let $(s_n)$ be a k-automatic binary sequence generated by the finite automaton $(Q,\Sigma_k,\delta,q_0,\Sigma_2,\tau)$.
Then
 \[
  C_2(s_n,N)\geq \frac{N}{k(|Q|+1)} \quad \text{for } N\geq k(|Q|+1).
 \]
\end{theorem}

\begin{proof}
 We can assume that $\delta(q_0,0)=0$, see \cite[Theorem 5.2.1]{alsh}.
 Let $\varphi:Q^*\rightarrow Q^*$ be defined by $\varphi(q)=\delta(q,0)\delta(q,1)\dots\delta(q,k-1)$ for $q\in Q$ and $\varphi(xy)=\varphi(x)\varphi(y)$. Let $\mathbf{w}=w_0w_1w_2\dots$ be an infinite word over $Q$ which is a fixed point of $\varphi$. Then $\delta(q_0,(n)_k)=w_{n}$ and $\tau(w_n)=s_n$ for all $n\geq 0$, see \cite[Proof of Theorem 6.3.2]{alsh}.
 
 Put
\[
  M=\left\lfloor \frac{\log(N/(|Q|+1))}{\log k}\right\rfloor \geq 1.
\]
By the pigeon hole principle, among the first $|Q|+1$ elements of $(w_n)$ there are two elements having the same value, say $w_i=w_j$, $0\leq i<j\leq |Q|$. Then for the $M$th iteration of $\varphi$, $\varphi^M:Q\rightarrow Q^{k^M}$ we have $\varphi^M(w_i)=\varphi^M(w_j)$, thus $w_{i\cdot k^M+l}=w_{j\cdot k^M+l}$ for $l=0, \dots
 k^M-1$ so $s_{i\cdot k^M+l}=s_{j\cdot k^M+l}$ for $l=0, \dots k^M-1$. Then
 \[
  C_2(s_n,N)\geq V(s_n,N,k^M,D)=\sum_{l=0}^{k^M-1} (-1)^{u_{i k^M+l}+u_{j k^M+l}}= k^M\geq \frac{N}{k(|Q|+1)}
 \]
with lags $D=(i\cdot k^M,j\cdot k^M)$.
\end{proof}

\subsection*{Examples}

The Thue-Morse sequence $(t_n)_{n\geq 0}$ can be defined by a finite automaton with two states, see Figure \ref{fig:TM}. Hence, Theorem \ref{thm:general} yields
\[
   C_2(t_n,N)\geq \frac{N}{6} \quad \text{for } N\geq 6.
\]

 \begin{figure}[h]
\begin{center}
\scalebox{1}{
\begin{tikzpicture}[auto,thick]
 \node (E) at (-1.5,0) [circle] {};
 \node (A) at (0,0) [circle, draw] {$A/0$};
 \node (B) at (3,0) [circle, draw] {$B/1$};
 \draw [->,bend left] (A) to  node {1} (B);
 \draw [->,bend left] (B) to  node {1} (A);
 \path (B) edge [loop above] node {0} (B);
 \path (A) edge [loop above] node {0} (A);
 \draw [->] (E) to  node {} (A);
\end{tikzpicture}
}
\end{center}
\caption{Automaton generating the Thue-Morse sequence.}
\label{fig:TM}
\end{figure}
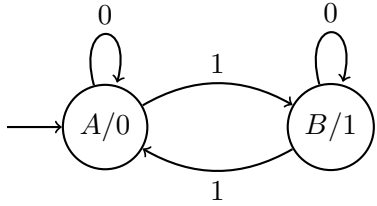

The Rudin-Shapiro sequence $(r_n)_{n\geq 0}$ can be defined by a finite automaton with four states, see Figure \ref{fig:RS}. Hence, Theorem \ref{thm:general} yields
\[
   C_2(t_n,N)\geq \frac{N}{10} \quad \text{for } N\geq 10.
\]

 \begin{figure}[h]
\begin{center}
\scalebox{1}{
\begin{tikzpicture}[auto,thick]
 \node (E) at (-1.5,0) [circle] {};
 \node (A) at (0,0) [circle, draw] {$A/0$};
 \node (B) at (3,0) [circle, draw] {$B/0$};
 \node (C) at (6,0) [circle, draw] {$C/1$};
 \node (D) at (9,0) [circle, draw] {$D/1$};

  \draw [->,bend left] (A) to  node {1} (B);
 \draw [->,bend left] (B) to  node {1} (C);
 \draw [->,bend left] (C) to  node {0} (D);

 \draw [->,bend left] (D) to  node {1} (C);
 \draw [->,bend left] (C) to  node {1} (B);
 \draw [->,bend left] (B) to  node {0} (A);

 \path (D) edge [loop above] node {0} (B);
 \path (A) edge [loop above] node {0} (A);
 \draw [->] (E) to  node {} (A);
\end{tikzpicture}
}
\end{center}
\caption{Automaton generating the Rudin-Shapiro sequence.}
\label{fig:RS}
\end{figure}
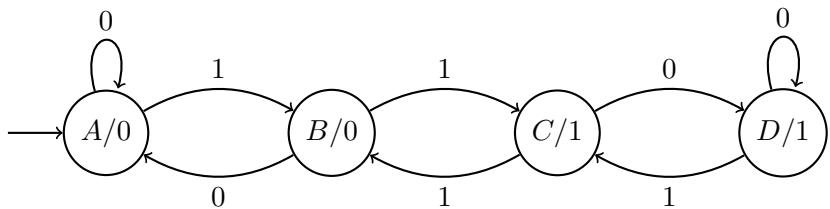

\subsection*{State complexity of binary sequences}

Theorem \ref{thm:general} allows to give lower bounds on the state complexity of binary sequences in terms of the correlation measure of order 2.

Let $k\geq 2$. Then the \emph{$N$th state complexity} $SC_k(s_n,N)$ of a sequence $(s_n)$ over $\F_2$ is the minimum of the number of states of finite $k$-automatons which generate the first $N$ elements. For example, the state complexity of the Thue-Morse sequence $(t_n)$ is $SC_2(t_n,N)=2$ for $N\geq 2$. By Theorem \ref{thm:general} we get lower bound on the $N$th state complexity of binary sequences.
\begin{corollary}\label{cor:SC}
 Let $(s_n)$ be a binary sequence. Then for all $k\geq 2$ we have
 \[
  SC_k(s_n,N)\geq\frac{N}{k\cdot C_2(s_n,N)}-1 \quad \text{for } N\geq 3.
 \]
\end{corollary}

As an example we can give  a lower bound on the state complexity of the \emph{Legendre sequence} $(l_n)$ defined by
\[
 l_n=
 \left\{
 \begin{array}{cl}
  0 & \text{if } \left(\frac{n}{p}\right)=1,\\
  1 & \text{otherwise},
 \end{array}
 \right.
\]
where $p>2$ is a prime number and $\left(\frac{\cdot}{p}\right)$ is the Legendre symbol modulo $p$. Mauduit and S\'ark\"ozy \cite{masa-legendre} proved that $C_2(l_n,N)\ll  p^{1/2}\log p$ for $N\leq p$ thus Corollary \ref{cor:SC} gives
\[
  SC_k(l_n,N)\gg\frac{N}{ k\cdot p^{1/2}\log p}-1 \quad \text{for } 3\leq N\leq p.
\]

\section*{Acknowledgment}
The authors would like to thank Christian Mauduit for helpful discussions. 




\end{document}